\documentclass[a4paper,10pt]{article}
\usepackage{amssymb,amsmath,amsthm}
\usepackage{fullpage}
\usepackage{hyperref}
\usepackage{eucal}
\usepackage{color}
\usepackage{stackrel}
\usepackage{graphicx}
\usepackage{mathtools}
\providecommand{\keywords}[1]
{
\small
  \textbf{\textit{Keywords:}} #1
}

\newcommand{\Real}{\mathbb{R}}
\newcommand{\C}{\mathbb{C}}

\newcommand{\sgn}{\mathop{\mathrm{sgn}}\nolimits}

\newcommand{\Dom}{\mathsf{D}}

\newcommand{\sii}{L^2}

\begingroup
    \makeatletter
    \@for\theoremstyle:=definition,remark,plain\do{%
        \expandafter\g@addto@macro\csname th@\theoremstyle\endcsname{%
            \addtolength\thm@preskip\parskip
            }%
        }
\endgroup
\newtheorem{Theorem}{Theorem}
\newtheorem{Lemma}{Lemma}
\newtheorem{Proposition}{Proposition}
\newtheorem{Corollary}{Corollary}

\theoremstyle{definition}
\newtheorem{Remark}{Remark}

\def\OMIT#1{}
\usepackage[normalem]{ulem}
\definecolor{DarkGreen}{rgb}{0,0.5,0.1} % David

\newcommand\soutD{\bgroup\markoverwith
{\textcolor{DarkGreen}{\rule[.5ex]{2pt}{1pt}}}\ULon}
\newcommand\soutP{\bgroup\markoverwith
{\textcolor{blue}{\rule[.5ex]{2pt}{1pt}}}\ULon}
\newcommand{\Hm}[1]{\leavevmode{\marginpar{\tiny%
$\hbox to 0mm{\hspace*{-0.5mm}$\leftarrow$\hss}%
\vcenter{\vrule depth 0.1mm height 0.1mm width \the\marginparwidth}%
\hbox to
0mm{\hss$\rightarrow$\hspace*{-0.5mm}}$\\\relax\raggedright #1}}}

\begin{document}
%

%-------%
% TITLE %
%-------%
%------------------------------------------%
%------------------------------------------%
\title{\textbf{\LARGE
On the 
regularity of  Dirac eigenfunctions 
}}
\author{Tuyen Vu}
\date{\small 
\begin{quote}
\begin{center}
Department of Mathematics, Faculty of Nuclear Sciences and 
Physical Engineering, \\ Czech Technical University in Prague, 
Trojanova 13, 12000 Prague 2, Czechia; \\
E-mail: thibichtuyen.vu@fjfi.cvut.cz.
\end{center}
\end{quote}
 April 2024
}
\maketitle
\vspace{-5ex} 
\begin{abstract}
\noindent
The article provides  proofs for the regularity of Dirac eigenfunctions, subject to MIT boundary conditions employed on various types of open sets ranging from smooth  ones to convex polygons in two dimensions, as well as on half-space and  smooth bounded domains in three-dimensional space. 
\end{abstract}

%
%------------------------------------------%
%------------------------------------------%

%---------------------%
\section{Motivation}
%---------------------%
%
%\noteD{Why fascinating? It is rather technical and important for applications.}%
One of the most important topics in partial differential equations is 
the study of regularity of solutions.
We begin by recalling
the elliptic equation employed
on a bounded open set $\Omega$, 
namely the inhomogeneous Dirichlet Laplacian: 
%\\
%
\begin{equation}\label{Dirichlet0} 
\left\{
\begin{aligned}
  -\Delta \psi &= f
  && \mbox{in} && \Omega 
  \,,
  \\
  \psi &= g 
  && \mbox{on} && \partial\Omega 
  \,.
\end{aligned}
\right.
\end{equation}
If $\Omega$ is of class $C^{1,1}$, 
$f \in L^2(\Omega)$ and $g \in H^\frac{3}{2}(\partial\Omega)$,
then there exists a unique solution $\psi$ in $H^2(\Omega)$ 
of problem \eqref{Dirichlet0}. 
The regularity requirements
can be reduced when altering the boundary conditions into Neumann boundary: 

\begin{equation}\label{Neumann} 
\left\{
\begin{aligned}
  -\Delta \psi + a_0 \psi &= f_1
  && \mbox{in} && \Omega 
  \,,
  \\
  \frac{\partial\psi}{\partial n} &= g_1 
  && \mbox{on} && \partial\Omega 
  \,,
\end{aligned}
\right.
\end{equation}
where~$n$ is the outward unit of~$\Omega$.
%\noteD{Grisvard seems to assume more on~$a_0$}%
In this case,
$f_1 \in \sii(\Omega), g_1 \in H^\frac{1}{2}(\partial\Omega)$ 
and $ a_0 \in L^\infty(\Omega), a_0 \geq \beta > 0$ a.e. in $\overline{\Omega}$
imply that there also exits a unique solution $\psi$
in  $H^2(\Omega)$ 
of problem \eqref{Neumann}.
Under the same conditions for $f_1, g_1$, the Robin Laplacian produces a unique solution in $ H^2(\Omega)$ of the 
boundary value problem
\begin{equation}\label{Robin} 
\left\{
\begin{aligned}
  -\Delta \psi &= f_1
  && \mbox{in} && \Omega 
  \,,
  \\
\frac{\partial\psi}{\partial n} + \alpha  \psi &= g_1 
  && \mbox{on} && \partial\Omega 
  \,,
\end{aligned}
\right.
\end{equation}
where $\alpha$ is a positive constant. 
All these mentioned results can be found in \cite{Grisvard}. 
The regularity conclusions play an important role in proving the self-adjointness of the Dirichlet, Neumann and Robin Laplace operators.

%\noteD{Use new paragraphs with indent. 
%Do not use extra vertical spaces.}%
%\noteD{Why more general? It is assumed $C^2$ in the theorems.}%
For unbounded domains, 
we also have the analogous results as the following theorems:

\begin{Theorem}\textbf{(regularity for the Dirichlet problem)}\label{D2}
Let $\Omega$ be an open set of
class $C^2$ with $\partial\Omega$ be bounded (or else $\Omega = \mathbb{R}^N_+)$. Let $f \in L^2(\Omega)$ and let $u \in H^1_0(\Omega)$
satisfy 
\begin{equation}
\int_\Omega \nabla u \nabla \phi + \int_\Omega u \phi = \int_\Omega f \phi \quad \quad \forall \, \phi \in H^1_0(\Omega).
\end{equation}

Then $u \in H^2(\Omega)$ and 
$\|u\|_{H^2} \leq C \|f\|_{L^2(\Omega)}$
, where $C$ is a constant depending only on
$\Omega$. Furthermore, if $\Omega$ is of class $C^{m+2}$ and $f \in H^m(\Omega)$, then\\
\begin{center}
$u \in H^{m+2}(\Omega)$ and 
$\|u\|_{
H^{m+2}} \leq C
\|f\|_{H^m}.$
\end{center}
In particular, if $f \in H^m(\Omega)$ with $m > \frac{N}{2}$, then $u \in C^2(\overline{\Omega})$. Finally, if $\Omega$ is of class
$C^\infty$ and if $f \in C^\infty(\overline{\Omega})$, then $u \in C^\infty(\overline{\Omega})$.

\end{Theorem}

\begin{Theorem}\textbf{(regularity for the Neumann problem)}\label{N2}
With the same assumptions
as in Theorem \ref{D2} one obtains the same conclusions for the solution of the Neumann
problem, i.e., for $u \in H^1(\Omega)$ such that 
satisfy 
\begin{equation}
\int_\Omega \nabla u \nabla \phi + \int_\Omega u \phi = \int_\Omega f \phi \quad \quad \forall \, \phi \in H^1(\Omega).
\end{equation}
\end{Theorem}

The proofs of Theorems \ref{D2} and \ref{N2} 
can be found   in \cite[Sec.~9.5]{Brezis} by using the method of translations or the technique of difference quotients due to L. Nirenberg.

Referring to relativistic effects in quantum materials, we study the smoothness of eigenfunctions of the Dirac operators under infinite-mass boundary conditions. Let $\Omega$ be a Lipschitz domain in $\Real^d, d = 2,3$. Putting 
\begin{equation}
N := \left\{\begin{array}{rcl}
	2 &\mbox{ if } d=2  \,,\\
	4 &\mbox{ if } d=3  \,,\\
\end{array}\right.
\end{equation}
The Dirac operator with MIT bag model is defined by
\begin{equation}
  \begin{aligned}
    T u &:= (-i \alpha \cdot \nabla + m \beta) u, \\
    \Dom(T) &:= \bigl\{ u \in H^1(\Omega; \mathbb{C}^N): 
        u|_{\partial \Omega} = -i \beta (\alpha \cdot n) u|_{\partial \Omega} \bigr\}.
  \end{aligned}
\end{equation}
Here $m$ is a non-negative constant that stands for the mass of a relativistic \mbox{(quasi-)}particle, $n
  :\partial\Omega\to\Real^d$ 
  is the outward unit normal of the set~$\Omega$,
 the notation $\alpha$ acts as $\alpha.w = \sum_{i=1}^d \alpha_i w_i$ for $ w \in \Real^d$ and  the $\C^{N\times N}$ matrices $\alpha_i$ can be defined as the Pauli spin matrices for $d=2$
\begin{equation} \label{def_Pauli_matrices}
  \alpha_1 = \sigma_1 := \begin{pmatrix} 0 & 1 \\ 1 & 0 \end{pmatrix}, \qquad
  \alpha_2 = \sigma_2 := \begin{pmatrix} 0 & -i \\ i & 0 \end{pmatrix}, \qquad 
  \beta = \sigma_3 := \begin{pmatrix} 1 & 0 \\ 0 & -1 \end{pmatrix}.
\end{equation}
and as
the Dirac matrices  for $d=3$
\begin{equation} \label{def_Dirac_matrices}
  \alpha_j = \alpha_j := \begin{pmatrix} 0 & \sigma_j \\ \sigma_j & 0 \end{pmatrix},  \quad 
   \quad \beta := \begin{pmatrix} I_2 & 0 \\ 0 & -I_2 
\end{pmatrix},
\end{equation}
with $j=1,2,3$ and $I_2$ is the $2 \times 2$-identity matrix.

The Dirac operators with
MIT (or infinite-mass) boundary conditions arise in quantum mechanics
in the study of  particles with spin $\frac{1}{2}$ or
in materials science of graphene and related structures.
 They have recently attracted a lot of attention \cite{Arrizibalaga-LeTreust-Raymond_2017, Arrizibalaga-LeTreust-Raymond_2018, Arrizibalaga-LeTreust-Mas-Raymond_2019, Behrndt, BBKO, J.B,2DK}.  
The characteristics of the eigenfunctions of the operator have important applications for the electronic and transport properties of these materials. 
%\soutD{It is an active area of research in mathematical physics and materials science.}
Moreover, the regularity of functions from the operator domain
is closely related to the self-adjointness.
%\soutD{In order to study the regularity of the eigenfunctions for Dirac operators, we can not omit the study of the self-adjointness of the operators because it is indeed related to the regularity of those.}
The Dirac operators with infinite-mass boundary conditions 
are known to be self-adjoint, at least if
the boundary~$\partial\Omega$ is $C^2$-smooth    
\cite{Benguria-Fournais-Stockmeyer-Bosch_2017b, J.B} in two and three dimensions, respectively, or if~$\Omega$ is a convex polygon in dimension two \cite{LeTreust-Ourmieres-Bonafos_2018}. For  merely Lipschitz domains, the self-adjointness is gained in a $H^\frac{1}{2}$ setting for both dimensions \cite{Behrndt}.
%\noteD{You miss the book chapter by Behrndt for merely Lips chitz domains.}%
In three dimensional space $\mathbb{R}^3$, 
the self-adjointness of the operators in polygonal boxes is still an open problem.

%\noteD{This paragraph about notations is not suited for the introduction.
%Move it to Preliminaries.}%

%\noteD{This paragraph is a partial repetition of above.
%Join and substantially improve.}%
The principle objective of the paper is to stimulate a systematic approach to the regularity of  eigenfunctions for self-adjoint Dirac operators under infinite-mass boundary conditions. One of the key features in our approach is to exploit a trace inverse map to establish a function with suitable regularity acting from the boundary to the given domain to utilise the elliptic regularity for Dirichlet and Neumann boundary conditions \cite{Evans, Grisvard, Brezis} in both cases of two and three dimensions. 
Another  point of the present paper is to apply the weak-convergence theory, the formula of square operator and 
the trace theorem in figuring out the smoothness of the eigenfunctions for Dirac operators employed on two-dimensional polygonal domains with some additional conditions. Furthermore, it completely proves that there exists a non-$H^2$-smooth eigenfunction of the operator in a two-dimensional rectangle.

%\noteD{Do not write numbers by hand, use cross-references.}%
The paper is structured into four sections. 
%\noteD{Do not describe what was already read.}%
Section \ref{sec:prelim} provides notations and definitions used in the paper.
Section \ref{Sec.3} presents the proof for the regularity of Dirac eigenfunctions  employed on both smooth and polygonal domains in dimension two. In the final section, we demonstrate that on  half-space and smooth bounded domains in three dimensions.
%------------------------------------------------------%
\section{Preliminaries}\label{sec:prelim}
%------------------------------------------------------%

Throughout the paper, we write $C^\infty_0(\Omega; \C^q),q \in \mathbb{N}^* $ for the space of $\C^q$-valued smooth functions with compact support in $\Omega \subset \mathbb{R}^N$ and we set
$$C^\infty(\overline{\Omega}; \C^q) := \Large\{ \psi \upharpoonright \Omega : f \in C^\infty_0(\mathbb{R}^N; \C^q)\Large\}$$
and for $k \in \mathbb{N}$,
$$C^k(\overline{\Omega}; \C^q) := \{ \psi\in C^k(\Omega; \C^q); D^\alpha \psi \,\mbox{ has a continuous extension on} \,\overline{\Omega} \,\mbox{ for all } \,\alpha  \mbox{ with} \,|\alpha| \leq k \}.$$
In addition, we recall the definitions of the Sobolev space $W^{m,p}(\Omega)$ for positive integers $m, p$ and  the fractional Sobolev space $ W^{s,p}$ for a real number $s$.
\begin{equation}
\begin{aligned}
W^{m,p}(\Omega) = \{ u\in L^p(\Omega) : \forall \alpha \,\mbox{with} \,|\alpha|\leq m, \exists\, g_\alpha \in L^p(\Omega) \,\mbox{such that}
 \int_\Omega u D^\alpha\varphi = (-1)^{|\alpha|} \int_\Omega g_\alpha \varphi \\
 \,\forall \varphi \in C^\infty_0(\Omega)\},
 \end{aligned}
\end{equation}
where we use the standard multi-index notation $\alpha = (\alpha_1, \alpha_2,..., \alpha_N)$ with $\alpha_i \geq 0$ an integer,
\begin{equation}
|\alpha| = \sum_{i=1}^N \alpha_i \, \quad \mbox{and} \, \quad D^\alpha \varphi = \frac{\partial^{|\alpha|}}{\partial x_1^{\alpha_1} \partial x_2^{\alpha_2}...\partial x_N^{\alpha_N}}.
\end{equation}
We set $D^\alpha u = g_\alpha$ and $H^m(\Omega) = W^{m,2}(\Omega)$ is a Hilbert space equipped with the inner product $$(u, v)_{H^m} = \sum_{i=0}^m (D^i u, D^i v)_{L^2}.$$
For a positive non-integer number $s$, denote the integer part of $s$ by $[s]$, we define $W^{s,p}$ for $p \in \mathbb{N}^*$ or $H^s$ for $p=2$,
\begin{equation}
\begin{aligned}
W^{s,p}(\Omega) = \{ u \in W^{[s],p} : \underset{|\alpha|=[s]}{\sup} \int_\Omega \int_\Omega \frac{|D^\alpha u(x)-D^\alpha u(y)|^p}{|x-y|^{N + p(s-[s])}} \,dx \,dy < \infty  \}.
\end{aligned}
\end{equation}
It is a Banach space for the norm $\|u\|_{W^{s,p}} = \|u\|_{W^{[s],p}} + \underset{|\alpha|=[s]}{\sup} (\int_\Omega \int_\Omega \frac{|D^\alpha u(x)-D^\alpha u(y)|^p}{|x-y|^{N + p(s-[s])}} \,dx \,dy)^{\frac{1}{p}}.$

\section{Regularity of the two-dimensional Dirac eigenfunctions}\label{Sec.3}
Let us recall the two-dimensional Dirac operator 
\begin{equation}\label{operator1}
  T := 
  \begin{pmatrix}
    m & -i (\partial_1-i\partial_2) \\
    -i(\partial_1+i\partial_2) & -m
  \end{pmatrix}
  \qquad \mbox{in} \qquad
  \sii(\Omega;\C^2)
  \,,
\end{equation}
with the boundary conditions are encoded in the operator domain 
 \begin{equation}\label{operator2}
  \Dom(T) :=  
  \left\{
  \psi = 
  \begin{psmallmatrix}
  \psi_1 \\ \psi_2 
  \end{psmallmatrix}
  \in H^1(\Omega;\C^2) : \ \psi_2 = i (n_1 + i n_2) \psi_1
  \mbox{ on } \partial\Omega
  \right\}
  .
\end{equation}
The operator~$T$ is not necessarily self-adjoint for a general Lipschitz set, but it can be made self-adjoint in a certain $H^{\frac{1}{2}}$ setting \cite{Behrndt}.
%\noteD{Give the related reference.}%
 Nevertheless, if the domain~$\Omega$ is $C^2$-smooth bounded \cite{Benguria-Fournais-Stockmeyer-Bosch_2017b}  or  convex polygonal \cite{LeTreust-Ourmieres-Bonafos_2018} then the self-adjointness of the operator is valid. Furthermore,  the eigenvalues of~$T$ are purely discrete and symmetrically distributed on the real axis. In these cases, the eigenfunction of~$T$ denoted by $u =  \begin{psmallmatrix}
  u_1 \\ u_2 
  \end{psmallmatrix}$ corresponding to an eigenvalue $\lambda$, can be characterised by the following system:
\begin{equation}\label{system} 
\left\{
\begin{aligned}
  -i(\partial_1-i\partial_2) u_2 &= (\lambda-m) u_1 \quad && \mbox{in} \quad \Omega \,,\\
  -i(\partial_1+i\partial_2) u_1 &= (\lambda+m) u_2  \quad && \mbox{in} \quad \Omega \,,\\
 u_2 &= i (n_1 + i n_2) u_1 && \mbox{ on } \partial\Omega .
\end{aligned}  
\right.
\end{equation}

The regularity of the  eigenfunctions of $T$ in two dimensions will be studied in this section. In particular, we will consider whether the eigenfunctions belong to the Sobolev space $H^k(\Omega)$ for some $k \geq 2$, which gauges the smoothness of functions in terms of their derivatives. This is essential for determining the characteristics and properties of eigenfunctions. The outcomes of the result will shed light on the smoothness of eigenfunctions for the Dirac operators with the MIT bag model depending on that of domains, including smooth ones and polygons in dimension two.

%----------------------------------------%
\subsection{Dirac eigenfunctions on smooth domains}\label{Sec.no}
%----------------------------------------%
%
Let~$\Omega$ denote a regular bounded domain on which the Dirac operator $T$ is defined.
Recall that $T$ is self-adjoint and its spectrum is purely discrete. By dint of self-adjointness, it is apparent that the eigenfunctions of $T$ are in $H^1(\Omega)$.
 In order to prove the further regularity of the eigenfunctions of $T$, 
 %in this subsection, we can suppose that $m=0$ due to the fact the proof of the main theorem is almost the same for both massless and positive mass cases .
 it is crucial to note that the process of separating the real and imaginary components allows us to apply elliptic regularity to complex-valued functions. We have the following theorem:
\begin{Theorem}\label{Th1}
%\noteD{Is $\Omega$ assumed to be bounded or not?
%The theorems you apply are formulated for bounded domains.}%
If $\partial\Omega$ belongs to the class $C^{k+2}$
with $k\in \mathbb{N}^*$, then every eigenfunction $u$ of $T$ is in the Sobolev space $H^{k+1}(\Omega)$.
\end{Theorem}
\begin{proof}
We commence by observing that the regularity
$C^{k+2}$ implies that the outward unit normal $n$ 
belongs to $C^{k+1}(\partial\Omega)$. 
Therefore, there exists a complex-valued function 
$\psi \in C^{k+1}(\overline{\Omega}) $ such that 
$$ \psi:= i(n_1+ i n_2) \quad  \mbox{on} \quad \partial\Omega. $$
Putting $\phi:= u_2 - \psi u_1$,
then $\phi$ satisfies the following system:
\begin{equation}\label{system1} 
\left\{
\begin{aligned}
  -i(\partial_1-i\partial_2) (\psi u_1 + \phi) &= (\lambda-m) u_1 \quad && \mbox{in} \quad \Omega \,,\\
  -i(\partial_1+i\partial_2) u_1 &= (\lambda+m) (\psi u_1 + \phi) \quad && \mbox{in} \quad \Omega \,,\\
 \phi &= 0 && \mbox{ on } \partial\Omega .
\end{aligned}  
\right.
\end{equation}
Since $u$ is an eigenfunction of $(T, \Dom(T))$ then we can deduce that $\Delta u = -(\lambda^2-m^2) u$.
As a result,
$\phi$ satisfies the second-order elliptic problem:
\begin{equation}
\left\{
\begin{aligned}
  -\Delta\phi - (\lambda^2-m^2) \phi &= -\Delta u_2 - (\lambda^2-m^2) u_2 + \Delta (\psi u_1) + (\lambda^2-m^2) (\psi u_1) \quad && \mbox{in} \quad \Omega \,,\\ 
 \phi &= 0 && \mbox{ on } \partial\Omega .
\end{aligned}  
\right.
\end{equation}
Equivalently,
\begin{equation}\label{system2} 
\left\{
\begin{aligned}
  -\Delta\phi - (\lambda^2-m^2) \phi &= (\lambda^2-m^2) (\psi u_1) + \Delta\psi u_1 -(\lambda^2-m^2) \psi u_1 + 2 \nabla\psi \nabla u_1 \quad && \mbox{in} \quad \Omega \,,\\ 
 \phi &= 0 && \mbox{ on } \partial\Omega .
\end{aligned}  
\right.
\end{equation}
The right-hand side of the first equation in the system \eqref{system2} belongs to $L^2(\Omega)$,  therefore applying the elliptic regularity
for the Dirichlet problem, 
we obtain that $\phi \in H^2(\Omega)$.

Reformulating  system \eqref{system1}, we have
\begin{equation}\label{eq3}
\left\{
\begin{aligned}
  -i(\partial_1-i\partial_2) (\psi u_1)  &= i(\partial_1-i\partial_2) \phi + (\lambda-m) u_1 \quad && \mbox{in} \quad \Omega \,,\\
  -i(\partial_1+i\partial_2) u_1 &= (\lambda+m) (\psi u_1 + \phi) \quad && \mbox{in} \quad \Omega \,.
\end{aligned}  
\right.
\end{equation}
Taking $\psi \in C^{k+1}(\overline{\Omega})$ into account, we deduce that 
\begin{equation}
\left\{
\begin{aligned}
  -i(\partial_1-i\partial_2) (\psi u_1)  & \in H^1(\Omega) \,,\\
  -i(\partial_1+i\partial_2) (\psi u_1) & \in H^1(\Omega) \,.
\end{aligned}  
\right.
\end{equation}
Therefore, $\psi u_1 \in H^2(\Omega)$ and then $u_2 \in H^2(\Omega)$.
Returning to the original boundary condition, 
we deduce that
%\noteD{Where? In $\Omega$ or on $\partial\Omega$?}%
\begin{equation}\label{u_1}
 u_1 = -i(n_1 -i n_2)u_2 \quad \quad \mbox{on} \quad \partial\Omega
 \end{equation}
Employing the trace theorem, we can derive that $u_2 \in H^\frac{3}{2}(\partial\Omega)$.
Applying the elliptic regularity for 
the Dirichlet  problem \eqref{Dirichlet0} to $u_1$, 
with the boundary value belonging  to $H^\frac{3}{2}(\partial\Omega)$, we deduce that $ u_1 \in H^2(\Omega)$.
Moreover, $\Delta\psi \in C^{k-1}(\overline{\Omega})$ due to $\psi \in C^{k+1}(\bar{\Omega})$,
 so the right-hand side of  \eqref{system2} presently belongs to $H^1(\Omega)$.

Using induction, we assume that $u \in H^{m}(\Omega; \C^2), \,\forall m \leq k$, 
and we will show that $u \in H^{m+1}(\Omega; \C^2)$.
Indeed, the right-hand side of \eqref{system2} is presently in $H^{m-1}(\Omega)$. 
Applying the elliptic regularity of Theorem \ref{D2}, we obtain that $\phi \in H^{m+1}(\Omega)$. Taking  \eqref{eq3} into account, we deduce that
\begin{equation}
\left\{
\begin{aligned}
  -i(\partial_1-i\partial_2) (\psi u_1)  & \in H^m(\Omega) \,,\\
  -i(\partial_1+i\partial_2) (\psi u_1) & \in H^m(\Omega) \,.
\end{aligned}  
\right.
\end{equation}
As a result, $\psi u_1 \in H^{m+1}(\Omega)$. In addition, $ u_2 = \phi + \psi u_1$ and $ \phi \in H^{m+1}(\Omega)$ as being proved previously. Consequently, we obtain $u_2 \in H^{m+1}(\Omega)$.  
%\noteD{Why?}%
Similarly, 
we also have $u_1 \in H^{m+1}(\Omega)$ due to \eqref{u_1}.
Therefore, $u_1, u_2 \in H^{k+1}(\Omega)$ by induction on $m$. 
Thus, the proof of Theorem \ref{Th1} is completed.
\end{proof}
\begin{Remark}
The validity of Theorem \ref{Th1} still holds for half-space or unbounded domains with bounded $\partial\Omega $. The proof remains unchanged when we choose  $ D^\alpha \psi \in L^\infty(\Omega)$ $\forall \, |\alpha|\leq k+1$. However, in the case of half-space domains, there is no point-spectrum of the operator due to the fact that there is no eigenfunction of the corresponding Laplacian. The author thanks D.~Krej\v{c}i\v{r}\'ik for pointing out this point of the spectrum for the Dirac problem in half spaces.

\end{Remark}

\begin{Corollary}

Let $\Omega$ be a $ C^\infty $-smooth bounded domain of $\mathbb{R}^2$ then every eigenfunction of $T$ is a $C^\infty$-smooth function on $\Omega$.
\end{Corollary}
\begin{proof}
The proof is obtained directly from Theorem \ref{Th1}.
\end{proof}

\begin{Remark}
If $\Omega$ bounded is of class $C^{1,1}$ and the extension $g_1$ of $g = i(n_1 + i n_2)$ on $\partial\Omega$ to $\Omega$ satisfies
  $\Delta g_1, g_1 \in L^\infty(\Omega; \C)$ and $\nabla g_1 \in L^\infty(\Omega; \C^2)$
then the eigenfunctions of $(T, \Dom(T))$ are also in $H^2(\Omega)$ by dint of Theorem \ref{Th1}.
\end{Remark}
 
The regularity of the eigenfunctions is still valid for the Dirac operator with an additional potential (which is similarly defined  in Section \ref{Sec.4}) by using the analogous arguments in the proofs of Theorems \ref{Th1}. To make the paper concise, we skip this problem in two dimensions.

%----------------------------------------%
 \subsection{Dirac eigenfunctions on polygons}\label{Sec.well}
%----------------------------------------%
%	
For further purposes, let us consider $\Omega$ is a two-dimensional polygon in $\mathbb{R}^2$ and $ (T, \Dom(T))$ denotes  the Dirac operator endowed with  MIT boundary conditions
%\noteD{Use align / aligned environment and align the lines.}%
%
\begin{equation}\label{operator1'}
\begin{aligned}
  T &:= 
  \begin{pmatrix}
    m & -i (\partial_1-i\partial_2) \\
    -i(\partial_1+i\partial_2) & -m
  \end{pmatrix}
  \qquad \mbox{in} \qquad
  \sii(\Omega;\C^2)
  \,,\\
  \Dom(T) &:=  
  \left\{
  \psi = 
  \begin{psmallmatrix}
  \psi_1 \\ \psi_2 
  \end{psmallmatrix}
  \in H^1(\Omega;\C^2) : \ \psi_2 = C_i \psi_1
  \mbox{ on } \gamma_i
  \right\}
  ,
  \end{aligned}
\end{equation}
where $\{\Gamma_i\}$ are sides of $\Omega$ and $C_i$ is a constant computed as the value of $i (n_1 + i n_2)$ on each side $\Gamma_i$ of the polygon. 
%\noteD{What is the meaning of this sentence?}%
Therefore, if we have $u \in C^\infty(\overline{\Omega}) $ and $ u \in \Dom(T)$ then it is apparent that $ \partial_i u \in \Dom(T) \,\forall i = 1,2$. This property is critical for the domain of operator employed on polygons.
Recall that if the polygon is convex then $(T, \Dom(T))$ is self-adjoint \cite{LeTreust-Ourmieres-Bonafos_2018}. As a result,
 the spectrum of $(T, \Dom(T))$ is purely discrete and the eigenfunctions are in $H^1(\Omega; \C^2)$. 

\begin{figure}[h]
  \begin{center}
  \includegraphics[scale=0.4]{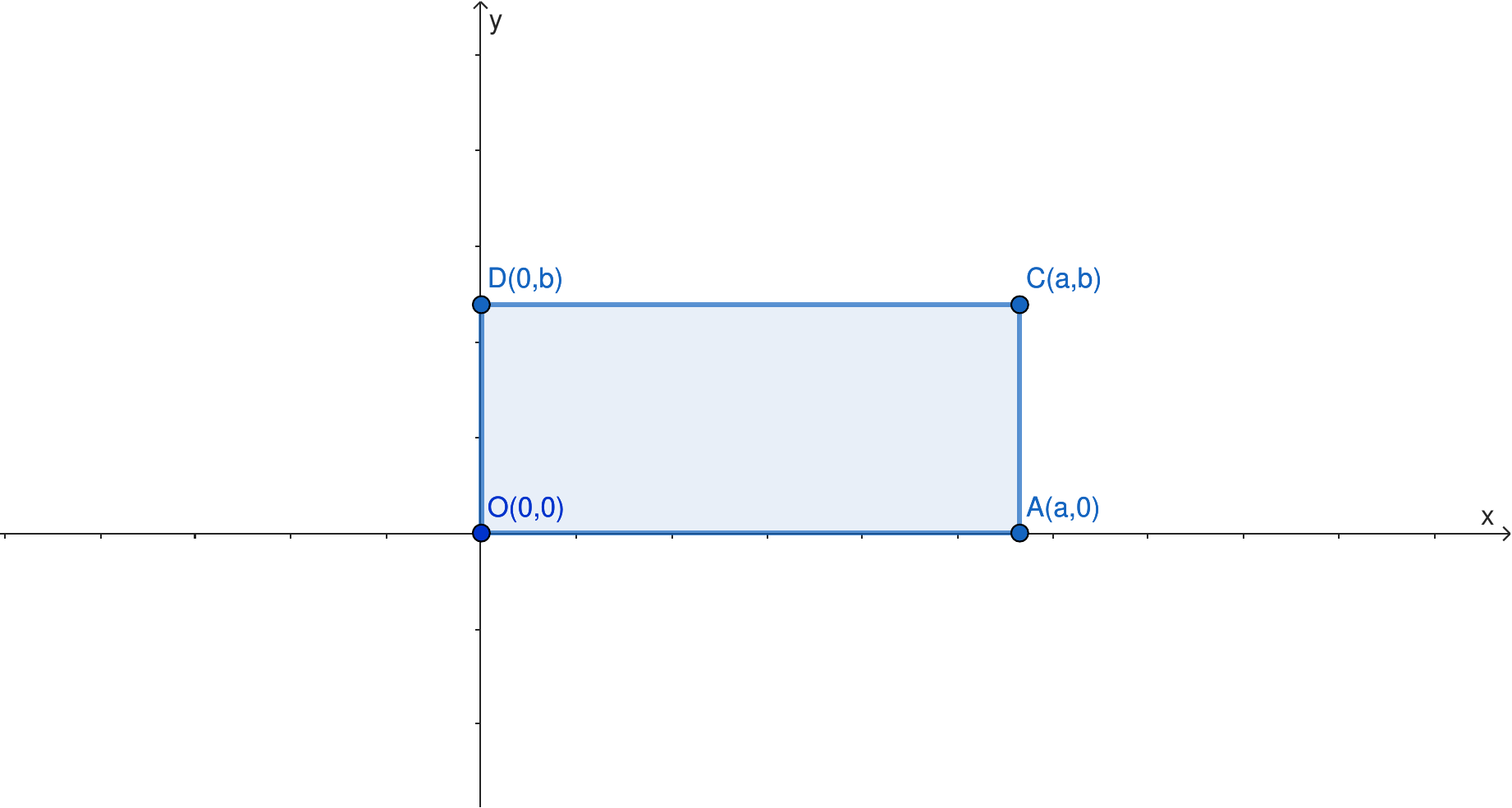}
  \caption{The rectangle $OACD$.
  \\
 % \txtD{Fonts in the figure are nor correct.
  %$x$-axis unnecessarily too long.}
  }
  \label{Fig}
  \end{center}
\end{figure}
 
The previous techniques applied to smooth domains are invalid for polygons. 
 Indeed,  there does not exist
 a function $\psi \in H^1(\Omega) $ such that $\Delta\psi \in L^2(\Omega)$ and $\psi = i(n_1+ i n_2)$ on the boundary of $\Omega$. 
To see it, let us for example consider the Dirac operator 
defined on a rectangle $OACD$, see Figure \ref{Fig}. 
We reformulate system \eqref{system}
\begin{equation}\label{ReOACD} 
\left\{
\begin{aligned}
  -i(\partial_1-i\partial_2) u_2 &= (\lambda -m) u_1
  && \mbox{in} \quad \Omega \,,
  \\
  -i(\partial_1+i\partial_2) u_1 &= (\lambda +m) u_2
  && \mbox{in} \quad \Omega \,,
  \\
  u_2 &= u_1 
  && \mbox{on} \quad OA \,,
  \\
  u_2 &= - u_1 
  && \mbox{on} \quad CD \,,
  \\
  u_2 &= iu_1 
  && \mbox{on} \quad AC \,,
  \\
  u_2 &= -iu_1 
  && \mbox{on} \quad OD
  \,.
\end{aligned}  
\right.
\end{equation}
To obtain the expected result, 
we look for
 a function $\psi \in H^1(\Omega)$ 
  satisfying  system 
\begin{equation}\label{system4} 
\left\{
\begin{aligned}
  \Delta\psi & \in L^2(\Omega) \,,\\ 
  \psi &= 1 
  && \mbox{on} \quad OA \,,
  \\
  \psi &= -1
  && \mbox{on} \quad CD \,,
  \\
  \psi &= i
  && \mbox{on} \quad AC \,,
  \\
  \psi &= -i
  && \mbox{on} \quad OD
  \,.
\end{aligned}  
\right.
\end{equation}
By estimating the integral on the boundary of the domain, we have
\begin{equation}
\begin{aligned}
\int_{OA} \int_{OD} \frac{|\phi(0,y)- \phi(x,0)|^2}{x^2+y^2} dy dx &= 2 \int_0^a \int_0^b \frac{1}{x^2+y^2} dy dx\\
&= 2 \int_0^a \frac{\arctan\frac{b}{x}}{x} \, dx\\
&\geq 2 \arctan\frac{b}{a} \int_0^a \frac{dx}{x} = +\infty.
\end{aligned}
\end{equation}
It yields that $\psi \notin H^\frac{1}{2}(\partial\Omega)$. As a result, $\psi \notin H^1(\Omega) $ due to the trace theorem. It contrasts with the imposed hypothesis.
Thus, the approach of extending the function defined by $i(n_1 + i n_2)$ on the boundary to the domain seems incapable.

For general polygons, we consider a sector $\Omega =
 OAB$ with
$
  O := ( 0,0) \, ,
  A := (a, 0) \, ,
  B : = ( a \cos\alpha, a \sin\alpha)  
.$ \\
 Let consider the Dirac problem, subject to the MIT boundary conditions
\begin{equation}\label{Sec1}
\left\{
\begin{aligned}
  -i(\partial_1-i\partial_2) u_2 &= f_1 
  && \mbox{in} \quad H^1(\Omega) \,,
  \\
  -i(\partial_1+i\partial_2) u_1 &= f_2
  && \mbox{in} \quad H^1(\Omega) \,,
  \\
  u_2 &= C_1 u_1 
  && \mbox{on} \quad OA \,,
  \\
  u_2 &= C_2 u_1 
  && \mbox{on} \quad OB \,,
  \\
  u_2 &= u_1 = 0 
  && \mbox{on} \quad \widehat{AB} \,.
\end{aligned}  
\right.
\end{equation}
 Here 
$  
   u = 
  \begin{psmallmatrix}
    u_1 \\ u_2 
  \end{psmallmatrix}
 \in W^{1,2}(\Omega;\C^2)$ and $ \C^* \ni C_1, C_2 $, which equals to $ i(n_1 + i n_2)$ on each corresponding side. 
Recall the elliptic problem for convex domains
\begin{Lemma}\label{Convex}
Let $\Omega$ be a convex , bounded, open subset of $\mathbb{R}^d, d \in \mathbb{N}^*$ then for each $f \in L^2(\Omega)$, there exists a unique $u \in H^2(\Omega)$, the solution of
\begin{equation} 
\left\{
\begin{aligned}
  -\Delta u &= f
  && \mbox{in} && \Omega 
  \,,
  \\
  u &= 0 
  && \mbox{on} && \partial\Omega 
  \,.
\end{aligned}
\right.
\end{equation}
\end{Lemma}
\begin{proof}
See \cite{Grisvard}.
\end{proof}
Denote $ M= ( \frac{1}{n}, 0), N= ( \frac{a}{n} \cos\alpha, \frac{a}{n} \sin\alpha )$ and $ M'= (\frac{a}{2n}, 0), N' = (\frac{a}{2n} \cos\alpha, \frac{a}{2n} \sin\alpha ), n \in \mathbb{N}^*$ then we have $ M, M' \in OA, N, N' \in OB$. \\
\begin{Lemma}\label{ReConvex}
If $u$ is a solution of problem \eqref{Sec1} then $u \in H^2(\Omega\setminus OMN)$.
\end{Lemma}

\begin{proof}
Take $\varphi \in C^\infty_0 ( \overline{\Omega}\setminus \overline{OM'N'}; \mathbb{R})$ and $\varphi = 1$ on $ \overline{\Omega}\setminus OMN$ then $\varphi = 0$ in $OM'N'$ and it deduces that there are $ f_1', f_2' \in H^1(\Omega) $ such that
\begin{equation}\label{Convex1} 
\left\{
\begin{aligned}
  -i(\partial_1-i\partial_2) (\varphi u_2) &= f'_1 
  && \mbox{in} \quad H^1(\Omega) \,,
  \\
  -i(\partial_1+i\partial_2)(\varphi u_1) &= f'_2
  && \mbox{in} \quad H^1(\Omega) \,,
  \\
  \varphi u_2 &= C_1 \varphi u_1 
  && \mbox{on} \quad OA \,,
  \\
 \varphi u_2 &= C_2 \varphi u_1 
  && \mbox{on} \quad OB \,,
  \\
  \varphi u_2 &= \varphi u_1 = 0 
  && \mbox{on} \quad \widehat{AB} \,.
\end{aligned}  
\right.
\end{equation}
We choose a function $g \in C^\infty( \partial\Omega)$ such that
\begin{equation} 
  g = \begin{cases}
   & C_1 \varphi  
   \quad \quad \mbox{on} \quad OA \,,
  \\
  & C_2 \varphi  
   \quad \quad \mbox{on} \quad OB \,,
  \end{cases} 
\end{equation}
Hence, there exists a function $g_1 \in C^\infty(\overline{\Omega}; \C)$ such that $g_1 = g$ on $\partial\Omega$. 

Putting $\phi := \varphi u_2 - g_1 u_1$ then $\phi \in H^1_0(\Omega)$ and
\begin{equation}
\Delta\phi = \Delta(\varphi u_2) - \Delta g_1 u_1 -2 \nabla g_1 \nabla u_1 - g_1 \Delta u_1 \in L^2(\Omega) .
\end{equation}
Hence, $\phi \in H^2(\Omega)$ due to Lemma \ref{Convex}. As a result,
$$ -i(\partial_1-i\partial_2) (\varphi u_2 - g_1 u_1) \in H^1(\Omega) .$$
Taking \eqref{Convex1} into account, we obtain $(\partial_1-i\partial_2) (g_1 u_1) \in H^1(\Omega)$. In addition, $(\partial_1+i\partial_2)u_1 \in H^1(\Omega)$ due to  \eqref{Sec1}, and thus, $(\partial_1+i\partial_2)(g_1 u_1) \in H^1(\Omega)$.
Consequently, we have
\begin{equation} 
\left\{
\begin{aligned}
  \partial_1 (g_1 u_1) & \in H^1(\Omega)   
  \,,
  \\
  \partial_2 (g_1 u_1) & \in H^1(\Omega)
  \,.
\end{aligned}
\right.
\end{equation}
Equivalently, we get $g_1 u_1 \in H^2(\Omega)$ and thus, $\varphi u_2 = \phi + g_1 u_1 \in H^2(\Omega)$. Because of the definition of the function $\varphi$, we obtain that $ u_2 \in H^2( \Omega\setminus OMN)$.

Reversely, we have 
\begin{equation}
\begin{aligned}
u_1 &= C_1^{-1} u_2 
  && \mbox{on} \quad OA \,,
  \\
  u_1 &= C_2^{-1} u_2 
  && \mbox{on} \quad OB \,,
  \\
  u_1 &= u_2 = 0 
  && \mbox{on} \quad \widehat{AB} \,.
\end{aligned}  
\end{equation}
Iterating the previous process, we also obtain that $ u_1 \in H^2 (\Omega\setminus OMN)$. Then, it concludes the proof of the lemma.
\end{proof}
\begin{Remark}
Putting $\Omega_n := \Omega\setminus \overline{OMN}$ then $ \Omega_n \xrightarrow{ n\rightarrow \infty} \Omega$. Hence,  $u \in H^2(\Omega\setminus V)$ for any neighbourhood $V$ of the vertex of the sector.

\end{Remark}

\begin{Theorem}\label{Theo.main}
Let $\Omega$ be a convex polygon in $\mathbb{R}^2$ and $V$ be any neighbourhood of all the set of vertices of $\Omega$ then the eigenfunctions of $T$ are in $ H^2(\Omega\setminus V)$.

\end{Theorem}
\begin{proof}
Using a partition of unity and Lemma \ref{ReConvex}, we can obtain the desired result.
\end{proof}

Establishing the regularity proof of solutions for elliptic problems in polygonal domains is exceedingly challenging and necessitates the utilisation of numerous techniques. Virtually, the significant results come from Dirichlet and Neumann boundary conditions.   In this paper, we  propose an approach supported by an additional hypothesis and a formula of the square 
of the Dirac
operator described in \cite{Tuyen}. 

\begin{Lemma}\label{L2}
Let $\Omega$ be a polygon in $\mathbb{R}^2$ and $U$ be the set of vertices of $\Omega$ then $C^\infty_0(\overline{\Omega}\setminus \{U\}) \cap \Dom(T)$  is a core of~$T$.
\end{Lemma}
%\begin{proof}
%\end{proof}

\begin{Lemma}\label{square-operator}
 Let $\Omega$ be a two-dimensional polygon and $(T, \Dom(T))$ is the Dirac operator subject to infinite mass boundary conditions then $\forall \,u\in \Dom(T)$,
\begin{equation*}  
  \|T u\|^2   = \|\nabla u\|^2_{\sii(\Omega)} + m^2 \|u\|^2_{\sii(\Omega)} + m \, \|\gamma u\|^2_{\sii ( \partial (\Omega)) } .
\end{equation*}
\end{Lemma}
The proofs can be found in \cite{Tuyen}. 
By virtue of Lemma \ref{L2}, we can find a sequence $\{u_{n}\}_{n\geq1} \subset C^\infty_0(\overline{\Omega}\setminus \{U\}) \cap \Dom(T)$ such that $u_{n} \xrightarrow{n \rightarrow \infty} u$ in $H^1(\Omega, \C^2)$.

\begin{Proposition}\label{Repoly}
Let $\Omega$ be a convex polygon in $\mathbb{R}^2$ and $u$ is an eigenfunction corresponding to an eigenvalue $\lambda$ of $T$ such that there exists $ C^\infty_0(\overline{\Omega}\setminus \{U\}) \cap \Dom(T)\ni u_{n} \xrightarrow{n \rightarrow \infty} u$ in $H^1(\Omega, \C^2)$,
 $\{\partial_1(\partial_1 + i \partial_2) u_{1n}\}_{n\geq1}$, $\{\partial_1(\partial_1 - i \partial_2) u_{2n}\}_{n\geq1}$ weakly converge in $L^2(\Omega)$  then $u  \in  H^2(\Omega)$.

\end{Proposition}
\begin{proof}
Recall that $u \in H^1(\Omega; \C^2)$ 
\begin{equation}\label{sys5}
\left\{
\begin{aligned}
  -i(\partial_1+i\partial_2)  u_1  &=  (\lambda + m) u_2 \quad && \mbox{in} \quad \Omega \,,\\
  -i(\partial_1-i\partial_2) u_2 &= (\lambda - m) u_1 \quad && \mbox{in} \quad \Omega \,.
\end{aligned}  
\right.
\end{equation}

By dint of Lemma \ref{L2}, we can find a sequence $\{u_{1n}\}_{n\geq1} \in C^\infty_0(\overline{\Omega}\setminus \{U\})$ such that $u_{1n} \xrightarrow{n \rightarrow \infty} u_1$ in $H^1(\Omega, \C)$. Therefore, we have $\partial_1 u_{2n} \xrightarrow{n\rightarrow \infty} \partial_1 u_{2}$ and $\partial_1 u_{1n}\xrightarrow{n\rightarrow \infty} \partial_1 u_{n}$ in $L^2(\Omega)$.

Using the hypotheses, there exists $\tilde{w}, w \in L^2(\Omega)$ such that
\begin{equation}
\begin{aligned}
-i \partial_1(\partial_1 + i \partial_2) u_{1n} - m \partial_1 u_{2n} \xrightharpoonup{n \rightarrow \infty} \tilde{w} \quad \mbox{in} \quad \sii(\Omega),\\
-i \partial_1(\partial_1 - i \partial_2) u_{2n} + m \partial_1 u_{1n} \xrightharpoonup{n \rightarrow \infty} w \quad \mbox{in} \quad \sii(\Omega).\\
\end{aligned}
\end{equation}
As a result, we have $\underset{n \in \mathbb{N}^*}{\sup} \|-i \partial_1(\partial_1 + i \partial_2) u_{1n}- m \partial_1 u_{2n} \|_{\sii(\Omega)} < +\infty$ and $\underset{n \in \mathbb{N}^*}{\sup} \|-i \partial_1(\partial_1 - i \partial_2) u_{2n} + m \partial_1 u_{1n}\|_{\sii(\Omega)} < +\infty$ due to uniform boundedness principle. 
To simplify formulae , we denote by $\|\cdot\|$ 
 the norm in $\sii(\Omega)$. 
 Since $u_n = (u_{1n}, u_{2n})^T \in \Dom(T)$ 
 then it follows that $\partial_1 u_n \in \Dom(T)$. 
 Thus,
 we obtain $\underset{n \in \mathbb{N}^*}{\sup}\|T \partial_1 u_n\| < +\infty$.
 
 Applying the formula of the operator square as described in 
 Lemma~\ref{square-operator}, 
 we have 
\begin{equation}
\begin{aligned}
\|T\partial_1 u_n\|^2 &= \|\partial_1^2 u_{1n}\|^2+ \|\partial_1\partial_2 u_{1n}\|^2 + \|\partial_1^2 u_{2n}\|^2+ \|\partial_1\partial_2 u_{2n}\|^2 + m^2 \|\partial_1 u_n\|^2 + m \|\partial_1 u_n \|^2_{L^2(\partial\Omega)} \,.
\end{aligned}
\end{equation}
Therefore, we obtain 
$$\underset{n \in \mathbb{N}^*}{\sup}\|\partial_1^2 u_{1n}\|^2 < +\infty \quad \mbox{and} \quad \underset{n \in \mathbb{N}^*}{\sup}\|\partial_1\partial_2 u_{1n}\|^2 < + \infty \,,$$
and
$$\underset{n \in \mathbb{N}^*}{\sup}\|\partial_1^2 u_{2n}\|^2 < +\infty \quad \mbox{and} \quad \underset{n \in \mathbb{N}^*}{\sup}\|\partial_1\partial_2 u_{2n}\|^2 < + \infty \,.$$

As a result, there exists sub-sequences convergent in $\sii(\Omega)$.
Taking a sub-sequence $\{u_{1n_k}\}_{k \geq 1} \subseteq \{u_{1n}\}_{n\geq 1} $ such that there exists $v \in \sii(\Omega)$ satisfying
$$ \partial_1\partial_2 u_{n_k} \xrightharpoonup{n \rightarrow \infty} v \in \sii(\Omega) ,$$
we have 
\begin{equation}
\begin{aligned}
(\partial_2 u_1, \partial_1 \phi) &= \underset{n_k\rightarrow +\infty}{\lim} \int_\Omega \overline{\partial_2 u_{1n_k}} \partial_1 \phi \, ,\\
&= - \underset{n_k\rightarrow +\infty}{\lim} \int_\Omega \partial_1\partial_2 \overline{u_{1n_k}} \phi \, ,\\
&= (-v,\phi) \,,
\end{aligned}
\end{equation}
valid for all $\phi \in C^\infty_0 (\Omega)$ and thus, we deduce that $\partial_1\partial_2 u_1 \in \sii(\Omega)$.

Employing the analogous arguments , we obtain that 
$\partial_1^2 u_1 \in \sii(\Omega)$. Moreover, $\Delta u_1 = -(\lambda^2 - m^2)u_1 \in H^1(\Omega)$. Therefore,
$\partial_2^2 u_1 \in \sii(\Omega)$ and thus,   $u_1 \in H^2(\Omega)$.
Analogously, we also have $u_2 \in H^2(\Omega)$. This concludes the proof.
\end{proof}

\begin{Remark}
Theorem \ref{Repoly} remains valid for non-convex polygons 
as long as the eigenfunctions of the Dirac operator 
  are in $H^1(\Omega)$. It is remarkable that there is no harmonic eigenfunction of the Dirac 
 operator under infinite-mass boundary conditions due to the fact that the square norm of the operator acting on any eigenfunction is positive.

  %Remembering that in these cases, we can not guarantee that the corresponding eigenfunctions are in $H^1(\Omega)$ due to not knowing the self-adjointness of the operator employed on such geometry domains.
\end{Remark}
 Furthermore, we would like to study on the existence of $H^2$-smooth eigenfunctions of the Dirac operators. Let us consider the Dirac problem, subject to Dirichlet boundary conditions
\begin{equation}\label{operator1''}
\begin{aligned}
  G &:= 
  \begin{pmatrix}
    m & -i (\partial_1-i\partial_2) \\
    -i(\partial_1+i\partial_2) & -m
  \end{pmatrix}
  \qquad \mbox{in} \qquad
  \sii(\Omega;\C^2)
  \,,\\
  \Dom(G) &:=  
  \left\{
  \psi = 
  \begin{psmallmatrix}
  \psi_1 \\ \psi_2 
  \end{psmallmatrix}
  \in H^1(\Omega;\C^2) : \ \psi_2 =  \psi_1 = 0
  \mbox{ on } \partial\Omega \,
  \right\}
   = H^1_0(\Omega; \C^2) 
  .
  \end{aligned}
\end{equation}
Due to $\Dom(G) \subsetneq \Dom(G^*)$, it follows that $ (G, \Dom(G))$ is  symmetric but not self-adjoint and if 
 $(\lambda,\psi)$ is an eigenvalue and eigenfunction pair of $(G, \Dom(G))$ then $(\lambda,\psi)$ is also an eigenpair of $(T, \Dom(T))$. Rigorously, showing the existence of the point spectrum of $G$ is not obvious, even if the domain is  smoother. Indeed, let us take a $C^\infty$-smooth domain, $F$ and a non-self-adjoint operator 
\begin{equation}
\begin{aligned}
  H_{F} :&= \begin{pmatrix}
    m & -i (\partial_1-i\partial_2) \\
    -i(\partial_1+i\partial_2) & -m
  \end{pmatrix}
  \qquad \mbox{in} \qquad
  \sii(F;\C^2)
  \,, \quad\\
  \Dom(H_{F}) & =
  \left\{
  \psi = 
  \begin{psmallmatrix}
  \psi_1 \\ \psi_2 
  \end{psmallmatrix}
  \in H_0^1(F;\C^2) \
  \right\}
  .
  \end{aligned}
\end{equation}
 Besides, the MIT bag models defined on $F$ act as
\begin{equation}
\begin{aligned}
  T_{F} &:= \begin{pmatrix}
    m & -i (\partial_1-i\partial_2) \\
    -i(\partial_1+i\partial_2) & -m
  \end{pmatrix}
  \qquad \mbox{in} \qquad
  \sii(F;\C^2)
  \,, \quad\\
  \Dom(T_{F}) & =
  \left\{
  \psi = 
  \begin{psmallmatrix}
  \psi_1 \\ \psi_2 
  \end{psmallmatrix}
  \in H^1(F;\C^2) : \ \psi_2 = i (n_1 + i n_2) \psi_1
  \mbox{ on } \partial\Omega
  \right\}
  .
  \end{aligned}
\end{equation}
It can be seen that $\Dom(H_F) \subset \Dom(T_F)$ then the point spectrum of $H_F$ is included in that of $T_F$. Moreover, in the case of $m=0, F = \mathbb{D}$-the unit disk in $\Real^2$, the eigenvalues $\lambda_n$ and eigenfunctions $u_n, n \in \mathbb{Z}$ of $T_{\mathbb{D}}$ are explicitly computed in polar coordinates \cite{HVD} :
\begin{equation}
u_n(r, \varphi)= \begin{pmatrix}
&e^{i n \varphi} J_n(|\lambda_n| r) \\
&i \sgn(\lambda_n) e^{i(n+1)\varphi} J_{n+1}(|\lambda_n| r)
\end{pmatrix} \,,
\end{equation}
where $J_n, J_{n+1}$ are the Bessel functions of order $n$ of first kind. However, $J_n$ and $ J_{n+1}$ have no common zeros, other than the origin due to the classical result \cite{Watson}. Therefore, $u_n$ does not vanish on the boundary of $\mathbb{D}$. Equivalently, $u_n$ is not an eigenfunction of $(H_\mathbb{D}, \Dom(H_\mathbb{D}))$ with $ n \in \mathbb{Z}, m=0$.

%We are a priori that there also exists  a non-$H^2$-smooth eigenfunction of $T$ in a two-dimensional polygonal domain due to the singularity of the boundary conditions of the operator domain.

...

%%%%%%%%%%%

%%%%%%%%%%
\section{Regularity of the three-dimensional Dirac eigenfunctions}\label{Sec.4}

%\begin{proof}[Proof of Proposition~\ref{prop:op1Dess}
%-------------------%
In this section, we consider the Dirac operators with MIT bag boundary conditions in three-dimensional $C^2$-smooth domains. Due to \cite{J.B}, the operators are self-adjoint.

To set the stage, let $\Omega$ be a domain with $C^2$-smooth boundary, we recall the definition of the operator defined 
in $L^2(\Omega; \mathbb{C}^4)$ 
\begin{equation} \label{def_MIT_op}
  \begin{aligned}
    T u &:= (-i \alpha \cdot \nabla + m \beta) u, \\
    \Dom(T) &:= \bigl\{ u \in H^1(\Omega; \mathbb{C}^4): 
        u|_{\partial \Omega} = -i \beta (\alpha \cdot n) u|_{\partial \Omega} \bigr\}.
  \end{aligned}
\end{equation}
and some elementary properties, see \cite{Arrizibalaga-LeTreust-Raymond_2018}:
For all $x,y \in \mathbb{R}^3$, $\eta :=\begin{pmatrix} 0 & I_2 \\ I_2 & 0 
\end{pmatrix},$ we have
\begin{equation}
\begin{aligned}
(\alpha x)\, (\alpha y)&= (x.y)1_4 + i \eta \, \alpha (x \times y)\\
\beta(\alpha x) &= - (\alpha x) \beta.
\end{aligned}
\end{equation}
Let $u \in H^1(\Omega; \C^4)$ be an eigenfunction  with respect to an eigenvalue $\lambda$ of $(T,\Dom(T))$.
 In view of
the properties of the Dirac matrices,  we have 
\begin{equation}
\begin{aligned}
&Tu = \lambda u \, ,\\
&\Longrightarrow (-i \alpha.\nabla + m \beta) u = \lambda I_4 u \, ,\\
&\Longrightarrow (-i \alpha.\nabla + m \beta)(-i \alpha.\nabla + m \beta) u = \lambda^2 I_4 u \, ,\\
&\Longrightarrow -\Delta u = (\lambda^2 - m^2) u \, .\\
\end{aligned}
\end{equation}
\\
The eigenfunction satisfies the following system:
\begin{equation}\label{system3} 
\left\{
\begin{aligned}
(\partial_1-i\partial_2) u_4 + \partial_3 u_3 &= i(\lambda -m) u_1 \quad && \mbox{in} \quad \Omega \,,\\
(\partial_1+i\partial_2) u_3 - \partial_3 u_4 &= i(\lambda -m) u_2 \quad && \mbox{in} \quad \Omega \,,\\
(\partial_1-i\partial_2) u_2 + \partial_3 u_1 &= i(\lambda +m) u_3 \quad && \mbox{in} \quad \Omega \,,\\
(\partial_1+i\partial_2) u_1 - \partial_3 u_2 &= i(\lambda +m) u_4 \quad && \mbox{in} \quad \Omega \,,\\
\end{aligned}  
\right.
\end{equation}
subject to boundary conditions,

\begin{equation}\label{system4} 
\left\{
\begin{aligned}
(-i n_1 - n_2) u_4 - i n_3 u_3 &=  u_1 \quad && \mbox{in} \quad \partial\Omega \,,\\
(-i n_1 + n_2) u_3 + i n_3 u_4 &= u_2 \quad && \mbox{in} \quad \partial\Omega \,,\\
(i n_1 + n_2) u_2 + i n_3 u_1 &=  u_3 \quad && \mbox{in} \quad \partial\Omega \,,\\
(i n_1 -n_2) u_1 - i n_3 u_2 &= u_4 \quad && \mbox{in} \quad \partial\Omega \,.\\
\end{aligned}  
\right.
\end{equation}

If $\Omega$ is a whole space $\mathbb{R}^3$ then the eigenfunctions are in
 $C^\infty(\mathbb{R}^3)$ by means of using the Fourier transform. Nevertheless, it is noticeable that the point spectrum of the operator is empty and the spectrum of $T$ is purely essential, $\sigma^{\mathbb{R}^3}(T) = \sigma_{\textup{ess}}^{\mathbb{R}^3}(T)=(- \infty, -m ] \cup [m, \infty)$, see \cite{T92}.

Let $V$ be in $C^\infty_0(\Omega; \mathbb{R})$ and consider an addition potential problem
 \begin{equation} \label{def_MIT_op2}
  \begin{aligned}
    D u &:= (-i \alpha \cdot \nabla + m \beta + V I_4) u, \\
    \Dom(D) &:= \bigl\{ u \in H^1(\Omega; \mathbb{C}^4): 
        u|_{\partial \Omega} = -i \beta (\alpha \cdot n) u|_{\partial \Omega} \bigr\}.
  \end{aligned}
\end{equation}
 \begin{Lemma}\label{Lemma-self}
 Let $\Omega$ be a $C^2$-smooth domain in $\mathbb{R}^3$, $V \in C^\infty_0(\Omega; \mathbb{R})$  then $(D, \Dom(D))$ is self-adjoint. In addition, if $\Omega$ is bounded, half-space or $\mathbb{R}^3$, then
  $(D, \Dom(D))$ and $(T, \Dom(T))$ have the same essential spectrum $\sigma_{\textup{ess}}^{\Omega}(D)=\sigma_{\textup{ess}}^{\Omega}(T)$. In particular, when $\Omega= \mathbb{R}^3$, $\sigma_{\textup{ess}}^{\mathbb{R}^3}(D)=(- \infty, -m ] \cup [m, \infty)$.
 
 \end{Lemma}
\begin{proof}
 The self-adjointness of $(D, \Dom(D))$ can be derived from the self-adjointness of $(T, \Dom(T))$ due to $ \Dom(D) = \Dom(T)$.
Indeed, it is obvious that $ (D, \Dom(D))$ is symmetric due to $(T, \Dom(T))$ and $ (V I_4, \Dom(D))$ are symmetric.
Take $\psi \in \Dom(D^*)$, for every $\phi \in \Dom(D) = \Dom(T)$, there exists a function $\eta \in L^2(\Omega; \mathbb{C}^2)$ such that 
\begin{equation}
( \psi, D \phi) = ( D \psi, \eta) \,.
\end{equation}
Equivalently, one has
\begin{equation}
(\psi, T \phi) + ( \psi, V I_4 \phi) = ( T \psi + V I_4 \psi, \eta) \,, 
\end{equation}
it deduces that 
\begin{equation}
( \psi, T \phi) = ( T \psi, \eta) \,.
\end{equation}
Therefore, $ \psi \in \Dom(T^*) = \Dom(T)$ and thus, $\Dom(D^*)= \Dom(D)$. It shows that $(D, \Dom(D))$ is self-adjoint.

Moreover,  if $\Omega$ is bounded or half-space then the trace map $\gamma: H^1(\Omega) \rightarrow H^\frac{1}{2}(\partial\Omega)$ is continuous. It deduces the closeness of the graph of  the following operator 
 \begin{equation} \label{def_MIT_op3}
  \begin{aligned}
  H  &: \Dom(H)\subset H^1(\Omega) \longrightarrow L^2(\Omega)\\
     &   \quad\quad \quad  u \longmapsto Hu = V I_4 u, \\
    \Dom(H) &:= \Dom(T) = \bigl\{ u \in H^1(\Omega; \mathbb{C}^4): 
        u|_{\partial \Omega} = -i \beta (\alpha \cdot n) u|_{\partial \Omega} \bigr\}.
  \end{aligned}
\end{equation}
Since the support of $V$ is compact, one deduces that $H$ is a compact operator or $D-T$ is compact. Therefore, $(D, \Dom(D))$ and $(T, \Dom(T))$ have the same essential spectrum due to the classical Weyl theorem \cite{Reed}.
In particular, one has $\sigma_{\textup{ess}}^{\mathbb{R}^3}(D)= \sigma_{\textup{ess}}^{\mathbb{R}^3}(T)=(- \infty, -m ] \cup [m, \infty)$. The proof of the Lemma is completed.

\end{proof}

\begin{Theorem}\label{Regularity hs}
Let $\Omega$ be any half-space domain in $\mathbb{R}^3$ then $(D, \Dom(D))$ is self-adjoint and if $u$ is 
an eigenfunction of $(D, \Dom(D))$ then $u \in C^\infty(\overline{\Omega})$. 
\end{Theorem}
There are perhaps alternative approaches to prove the regularity result, such as employing reflections or extension maps. However, it is definitely non-trivial for non-zero mass cases due to the fact that the normal extension based on the boundary conditions can not satisfy the equations of the eigenfunction, even it still holds for $m=0$.

In this paper, we demonstrate the desired result by establishing the elliptic regularity, 
we recall the  Neumann problem
\begin{equation}\label{N3} 
\left\{
\begin{aligned}
  -\Delta \psi +  \psi &= f
  && \mbox{in} && \Omega 
  \,,
  \\
  \frac{\partial\psi}{\partial n} &= g 
  && \mbox{on} && \partial\Omega 
  \,,
\end{aligned}
\right.
\end{equation}
where $\Omega$ is a three-dimensional domain, 
    $f$ belongs to $L^2(\Omega)$ and $g \in H^\frac{1}{2}(\partial\Omega)$.
We have already mentioned
that if $\Omega$ is a bounded regular domain then $\psi \in H^2(\Omega)$, see \cite{Grisvard}. Now we prove that the result still holds for half-space domains.  

%\noteD{Define half-space domain.}%

\begin{Lemma}\label{Th6}
Let $\Omega$ be any half space domain in $\mathbb{R}^3$, $ f \in L^2(\Omega)$ and $g \in H^\frac{1}{2}(\partial\Omega)$, system \eqref{N3} has a unique solution in $H^1(\Omega)$ and then $\psi \in H^2(\Omega)$.
\end{Lemma}
\begin{proof}
We prove the result in the case $\Omega = \mathbb{R}^2\times(0, \infty)$. For other half-space domains, the result can be obtained by using the same argument.
Since $\psi \in H^1(\Omega)$ satisfying system \eqref{N3} then 
%\noteD{Bars (complex conjugates) missing.}%
%
\begin{equation}\label{V1}
\int_\Omega \overline{\nabla \psi} \nabla \phi + \int_\Omega \overline{\psi} \phi = \int_\Omega \overline{f} \phi + \int_{\partial\Omega} \overline{g} \phi \quad \quad \forall \, \phi \in H^1(\Omega).
\end{equation}
Due to the trace theorem, there exists a unique continuous trace map 
  $H^1(\Omega) \rightarrow H^\frac{1}{2}(\partial\Omega)$ \cite[Sec. 1.5]{Grisvard} 
and   
a constant $C$ such that\\
\begin{equation*}
\|\phi\|_{H^\frac{1}{2}(\partial\Omega)} \leq C \|\phi\|_{H^1(\Omega)} \quad \forall  \phi \in H^1(\Omega).
\end{equation*}
As a result, we can deduce that problem \eqref{V1} 
has unique solution in $H^1(\Omega)$ due to the Lax-Milgram theorem.

On the other hand, there also exists a right continuous inverse of the trace map 
such that for any $g \in H^\frac{1}{2}(\partial\Omega)$, there is an extension $g_1 \in H^2(\Omega)$ such that $\frac{\partial g_1}{\partial n} = g $ on $\partial\Omega$.  Putting $ u:= \psi - g_1$ then $u$ satisfies
\begin{equation}\label{N4} 
\left\{
\begin{aligned}
  -\Delta u +  u &= f - g_1 - \Delta g_1 
  && \mbox{in} && \Omega 
  \,,
  \\
  \frac{\partial u}{\partial n} &= 0 
  && \mbox{on} && \partial\Omega 
  \,.
\end{aligned}
\right.
\end{equation}
Employing the Neumann regularity presented in Theorem \ref{Neumann}, we obtain that $u \in H^2(\Omega)$ and thus we also have $\psi \in H^2(\Omega)$. The proof has been completed.
\end{proof}

Now we are in a position to prove Theorem \ref{Regularity hs}.
\begin{proof}[Proof of Theorem~\ref{Regularity hs}] 
The self-adjointness can directly be deduced from Lemma \ref{Lemma-self}.

We still consider $\Omega = \mathbb{R}^2\times (0,\infty)$, 
let $u$ be an eigenfunction  corresponding to an eigenvalue $\lambda$ of $(D,\Dom(D))$ then
system  \eqref{system4} reads,
\begin{equation}\label{system4'} 
\left\{
\begin{aligned}
 i u_3 &=  u_1 \quad && \mbox{on} \quad  \mathbb{R}^2\times \{0\} \,,\\
-i  u_4 &= u_2 \quad && \mbox{on} \quad \mathbb{R}^2\times \{0\} \,,\\
-i u_1 &=  u_3 \quad && \mbox{on} \quad  \mathbb{R}^2\times \{0\} \,,\\
 i  u_2 &= u_4 \quad && \mbox{on} \quad \mathbb{R}^2\times \{0\} \,.\\
\end{aligned}  
\right.
\end{equation}
Using the regularity  stated in Theorem \ref{D2}, we obtain that
\begin{equation}\label{B1}
\begin{aligned}
 u_1 -i u_3 \in H^p (\mathbb{R}^2\times (0,\infty)) \,\forall \,p \in \mathbb{N} ,\\ 
 u_2 + i u_4 \in H^p (\mathbb{R}^2\times (0,\infty)) \,\forall \, p \in \mathbb{N} .
 \end{aligned}
 \end{equation}
Moreover, we have
\begin{equation}\label{system3*} 
\left\{
\begin{aligned}
(\partial_1-i\partial_2) u_4 + \partial_3 u_3 &= i(\lambda -m - V)  u_1 \quad && \mbox{in} \quad \Omega \,,\\
(\partial_1+i\partial_2) u_3 - \partial_3 u_4 &= i(\lambda -m - V) u_2 \quad && \mbox{in} \quad \Omega \,,\\
(\partial_1-i\partial_2) u_2 + \partial_3 u_1 &= i(\lambda +m - V) u_3 \quad && \mbox{in} \quad \Omega \,,\\
(\partial_1+i\partial_2) u_1 - \partial_3 u_2 &= i(\lambda +m -V) u_4 \quad && \mbox{in} \quad \Omega \,.\\
\end{aligned}  
\right.
\end{equation}
Reformulating partly system \eqref{system3*}, we have
\begin{equation}\label{system3'} 
\left\{
\begin{aligned}
(\partial_1-i\partial_2)(u_2+ i u_4) + \partial_3 (u_1 + i u_3) &= i(\lambda +m -V)u_3 - (\lambda -m - V)  u_1 \quad && \mbox{in} \quad \Omega \,,\\
(\partial_1+i\partial_2) (u_1 -i u_3) - \partial_3 (u_2 - i u_4) &= i(\lambda +m -V) u_4+ (\lambda -m - V) u_2 \quad && \mbox{in} \quad \Omega \,.\\
\end{aligned}  
\right.
\end{equation}
As a result, one has
$$ \partial_3(u_1 + i u_3) \in H^1(\Omega) \, ,$$
$$ \partial_3(u_2 - i u_4) \in H^1(\Omega) \, .$$
Taking \eqref{B1} into account, we deduce that $\partial_3 u_i \in H^1(\Omega)$ for every $i= 1,2,3$.

Applying the trace theorem, we have $\partial_3 u_i \in H^\frac{1}{2}(\partial\Omega)$. Moreover,
the outward unit normal of $\Omega$  reads $ n = (0,0, -1)$ and thus $\frac{\partial u_i}{\partial n} \in H^\frac{1}{2}(\partial\Omega) $, for every $ i =1,2, 3$.
%\noteD{What is $\overline{1,3}$?}% 
Hence, $ u_i$ satisfies 

\begin{equation}\label{N5} 
\left\{
\begin{aligned}
  -\Delta u_i &= ((\lambda- V)^2 -m^2) u_i 
  && \mbox{in} && \Omega 
  \,,
  \\
  \frac{\partial u_i}{\partial n} & \in H^\frac{1}{2}(\partial\Omega) 
  \,.
\end{aligned}
\right.
\end{equation}
%
%Here $f_i = \sum_{j=1}^4 g_j(V) u_j    \in H^1(\Omega)$, $g_j$ is a derivative operator, for instance $f_4 = -i (\partial_1 + i \partial_2)V u_1 + i \partial_3 V u_2$. 
It is apparent that $u_i \in H^2(\Omega)$ due to Lemma \ref{Regularity hs}.
Thus, we have proved that $u_i \in H^m(\Omega)$ for every $ i= 1, 2, 3$ with $m=2$. Now we use induction on $m$ to show that $u \in H^k (\Omega)$ $\forall \, k \in \mathbb{N}$, we suppose that $u_i \in H^m(\Omega)$ for every $ i= 1, 2, 3$ and $ m \geq 2$.

Denote
$D^k := \partial_1^{\sigma_1} \partial_2^{\sigma_2} \partial_3^{\sigma_3}$ such that $ \sum_{i=1}^3 \sigma_i = k, k, \sigma_i \in \mathbb{N}$. We remark that the notation is purely symbolic as it represents different derivative operators but it does not impact the accuracy of the proof.

Exploiting system \eqref{system3'}, 
we deduce that both sides of the following equations are in $H^1(\Omega)$ and satisfy
\begin{equation}\label{system3''} 
\left\{
\begin{aligned}
(\partial_1-i\partial_2)& D^{m-1} (u_2 + i u_4) + \partial_3 D^{m-1} (u_1 + i u_3) = \,D^{m-1}[i(\lambda +m - V) u_3]- D^{m-1}[ (\lambda -m - V) u_1]  \quad \mbox{in}\quad\Omega \,,\\
(\partial_1+i\partial_2)& D^{m-1}(u_1- i u_3) - \partial_3 D^{m-1} (u_2 - i u_4) = \,D^{m-1}[i(\lambda +m - V)u_4] + D^{m-1}[(\lambda -m - V)  u_2]   \quad \mbox{in}\quad\Omega \,.
\end{aligned}  
\right.
\end{equation}
By virtue of \eqref{B1}, one has
\begin{equation}\label{B2}
\begin{aligned}
 D^k(u_1 -i u_3) \in H^p(\mathbb{R}^2\times (0,\infty))  ,\\ 
 D^k(u_2 + i u_4) \in H^p (\mathbb{R}^2\times (0,\infty)), \quad \,\forall k, p \in \mathbb{N} .
 \end{aligned}
 \end{equation}
By virtue of the inductive hypothesis, \eqref{system3''} and \eqref{B2}, we deduce that 
$$
\partial_3 D^{m-1} u_i \in H^1(\Omega) \, \forall i= 1, 2, 3.\\$$ 
 Using the trace theorem we have 
 \begin{equation}
\begin{aligned}
\partial_3 D^{m-1} u_i \in H^\frac{1}{2}(\partial\Omega) \, \forall i= 1, 2, 3.\\
 \end{aligned}
 \end{equation}
 Therefore, $\{D^{m-1} u_i\}_{i= 1, 2, 3} $ are solutions of the following systems:
\begin{equation}\label{N6} 
\left\{
\begin{aligned}
  -\Delta D^{m-1} u_i &= D^{m-1}[((\lambda- V)^2 -m^2)  u_i]  
  && \mbox{in} && \Omega 
  \,,
  \\
  \frac{\partial D^{m-1} u_i}{\partial n} & \in H^\frac{1}{2}(\partial\Omega) 
  \,.
\end{aligned}
\right.
\end{equation} 
 Taking Lemma \ref{Th6} into account when the right-hand side of the first equation in  \eqref{N6} is presently in $L^2(\Omega)$ , we obtain that $D^{m-1} u_i \in H^2(\Omega) \quad \forall i= 1, 2, 3$, it implies that $u_i \in H^{m+1}(\Omega)$. By induction on $m$, we obtain that $u \in H^k (\Omega)$ $\forall \, k \in \mathbb{N}$ and as a result, $u \in C^\infty (\overline{\Omega})$.
The proof argument remains unchanged for other half-space domains, this concludes the proof.
 \end{proof}
Now, we investigate the spectral problem for half-space domains:
\begin{Proposition}
Let $\Omega$ be any half-space in $\mathbb{R}^3$ then  $\sigma_{\textup{ess}}^{\Omega}(D)=\sigma^{\Omega}(T)$ and
there is no point spectrum of $(T, \Dom(T))$.
\end{Proposition}
 \begin{proof}
If we suppose that there exists an eigenfunction $u \in H^1(\Omega)$ of $(T, \Dom(T))$ then it follows  $ u \in C^\infty(\overline{\Omega})$ due to Theorem \ref{Regularity hs}. Without of loss generality, we can assume $\Omega = \mathbb{R}^2\times (0,\infty)$.
 As a result, we deduce that $u_1 - i u_3 := X$ is a smooth function on $\Omega$, vanishes on the boundary $\partial\Omega$ and satisfy \\
 \begin{equation}\label{E1} 
\left\{
\begin{aligned}
  \Delta X + (\lambda^2 - m^2) X &= 0 
  && \mbox{in} && \Omega 
  \,,
  \\
  X &= 0 
  && \mbox{on} && \mathbb{R}^2\times \{0\}
  \,.
\end{aligned}
\right.
\end{equation}
 By dint of $X \in C^\infty(\overline{\Omega})$, we can extend $X$ to a smooth function $\tilde{X} \in C^\infty(\mathbb{R}^3)$ such that 
 $$\tilde{X}(x,y,z) = \begin{cases}
 & X(x,y,z) \, \quad \mbox{if} \quad (x,y,z) \in \Omega \, ,\\
 & X(x,y,-z) \, \quad \mbox{if}\quad (x,y,z) \in \mathbb{R}^2\setminus\Omega \, .
 \end{cases}
 $$
 Therefore, $ 0 \neq \tilde{X} \in L^2(\mathbb{R}^3)\cap C^\infty(\mathbb{R}^3)$ and satisfies the following equation
\begin{equation}\label{equa1}
\Delta \tilde{X} + (\lambda^2 - m^2) \tilde{X} = 0  \quad \mbox{in} \quad \Real^3 \,.
\end{equation}
It contrasts with the fact that there is no eigenfunction of the Laplacian in $\Real^3$ by means of the Fourier transform. Therefore, there is no point-spectrum or eigenfunction of $(T, \Dom(T))$ on any half-space or $\sigma_{\textup{ess}}^{\Omega}(T)=\sigma^{\Omega}(T)$.

Due to Lemma \ref{Lemma-self}, one has $\sigma_{\textup{ess}}^{\Omega}(D)=\sigma_{\textup{ess}}^{\Omega}(T)$. Consequently, $\sigma_{\textup{ess}}^{\Omega}(D)=\sigma^{\Omega}(T)$. It concludes the proof.

 %We denote $\mathcal{F}$, the Fourier transform and $\mathcal{F}\tilde{X}({\xi}) := Y(\xi)$. Applying the Fourier transform on equation $\eqref{equa1}$, we have \\
% $$ (|\xi|^2 + \lambda^2 - m^2) Y(\xi) = 0.$$
% It follows that $ Y(\xi) = 0$ and thus $\tilde{X} = 0 \in \mathbb{R}^3$. It makes a contradiction. \\
% To sum up, there is no point-spectrum or eigenfunction of $(T, \Dom(T))$ on any half-space.

 \end{proof}
 For smooth-bounded domains, it seems hard to get the desired result if one only focuses on the Dirichlet regularity. Nevertheless, with reasonable manipulations and arguments, we can achieve the result by exploiting the Neumann and Dirichlet regularities, along with the trace extension and induction arguments. 
 
 %\end{Remark}
 \begin{Theorem}\label{Main}
 Let $\Omega$ be a $C^k$ smooth bounded domain in $\mathbb{R}^3$, $ k \in \mathbb{Z}, k\geq 3$  and $u$ is an eigenfunction corresponding to an eigenvalue $\lambda$ of $(T, \Dom(T))$  then $u \in H^{k-1}(\Omega)$.
 
 \end{Theorem}
 \begin{proof}
Since $\Omega$ is of class $C^k$ then $n $ is of class $C^{k-1}(\partial\Omega)$. Therefore, there exists functions $f,g \in C^{k-1}(\overline{\Omega})$ such that 
  $ f := -in_3, g := -in_1 -n_2$ on $\partial\Omega$ then $ \overline{g} = in_1 -n_2$ on $\partial\Omega$.
 
If  $ u$ is an eigenfunction of $(T, \Dom(T))$ then $u$ satisfies systems \eqref{system3} and \eqref{system4} and thus, we have  
\begin{equation}\label{system40} 
\left\{
\begin{aligned}
g u_4 +f u_3 - u_1 &=  0 \quad && \mbox{in} \quad \partial\Omega \,,\\
-\overline{g} u_3 -f u_4 - u_2 &= 0 \quad && \mbox{in} \quad \partial\Omega \,,\\
-g u_2 -f u_1 - u_3 &=  0 \quad && \mbox{in} \quad \partial\Omega \,,\\
\overline{g} u_1 +f  u_2- u_4 &= 0 \quad && \mbox{in} \quad \partial\Omega \,.\\
\end{aligned}  
\right.
\end{equation}
In addition, if $h \in C^2(\overline{\Omega})$ then $$ \Delta(h u_i) = \Delta h u_i + u_i \Delta u_i + 2 \nabla h \nabla u_i \in L^2(\Omega),$$ for every $i = 1, 2, 3$.
 Employing Theorem \ref{D2} and $\Delta u_i = -(\lambda^2 - m^2) u_i$, $\forall i = 1, 2, 3$, we deduce that 
 \begin{equation}\label{system4*} 
\left\{
\begin{aligned}
u_1 - g u_4 - f u_3   && \mbox{in} \quad H^2(\Omega) \,,\\
u_2 +\overline{g} u_3 + f u_4  && \mbox{in} \quad H^2(\Omega) \,,\\
u_3 +g u_2 + f u_1  && \mbox{in} \quad H^2(\Omega) \,,\\
u_4 -\overline{g} u_1 - f  u_2  && \mbox{in} \quad H^2(\Omega) \,.\\
\end{aligned}  
\right.
\end{equation}
% 
 %Taking $ u_i \in H^1(\Omega)$ for every $i = 1, 2, 3$, we obtain that $ u_i \in H^\frac{1}{2}(\partial\Omega)$ due to the trace theorem.
  It yields $$  (\partial_1 + i \partial_2) u_1 - f (\partial_1 + i \partial_2) u_3 - g (\partial_1 + i \partial_2) u_4 \in H^1(\Omega).$$ 
 %and thus, $$ - \partial_3 u_1 + f \partial_1 u_3 + g \partial_3 u_4 \in H^\frac{1}{2}(\partial\Omega).$$
 Taking \eqref{system3} into account, we have 
 $$(\partial_1 + i \partial_2) u_1 = \partial_3 u_2 + i(\lambda+m) u_4,$$
 $$(\partial_1 + i \partial_2) u_3 = \partial_3 u_4 + i(\lambda-m) u_2,$$
 From \eqref{system4*}, we obtain 
 $$ \partial_3 u_2 + \overline{g} \partial_3 u_3 + f \partial_3 u_4 \in H^1(\Omega) .$$
Hence, we deduce that
$$ - \overline{g}  \partial_3 u_3 - f \partial_3 u_4 - f \partial_3 u_4 - g (\partial_1 + i \partial_2) u_4 \in H^1(\Omega) .$$
Moreover, we derive $ \partial_3 u_3$ from system \eqref{system3}, we have 
 $$ \partial_3 u_3 = -(\partial_1 - i \partial_2) u_4 + i(\lambda - m) u_1 . $$
 It yields
 \begin{equation}
 \overline{g} (\partial_1 - i \partial_2) u_4 -2 f \partial_3 u_4 - g (\partial_1 + i \partial_2) u_4 \in H^1(\Omega) .
 \end{equation}
 Applying the trace theorem, we obtain
 \begin{equation}
 (i n_1 - n_2) (\partial_1 - i \partial_2) u_4 + 2 i n_3 u_4 + (i n_1 + n_2) (\partial_1 + i \partial_2) u_4 \in H^\frac{1}{2}(\partial\Omega) .
 \end{equation}
 Equivalently, we obtain that
 \begin{equation}
( n_1 \partial_1 + n_2 \partial_2 + n_3 \partial_3) u_4 \in H^\frac{1}{2}(\partial\Omega) .
 \end{equation}
 It implies that $\frac{\partial u_4}{\partial n} \in H^\frac{1}{2}(\partial\Omega).$
As a result, we deduce that 
 \begin{equation} 
\left\{
\begin{aligned}
  -\Delta u_4 &= (\lambda^2 -m^2) u_4 
  && \mbox{in} && \Omega 
  \,,
  \\
  \frac{\partial u_4}{\partial n} & \in H^\frac{1}{2}(\partial\Omega) 
  \,.
\end{aligned}
\right.
\end{equation}
 It implies that $ u_4 \in H^2(\Omega)$ due to the elliptic regularity for Neumann boundary  described in system \eqref{Neumann}.
 Using the analogous arguments for $u_1, u_2$ and $u_3$, we obtain $u_i \in H^2(\Omega)$ for every $i = 1, 2, 3.$
 
 Now we use the induction to prove the desired result, we suppose it holds for $ j < k-1$, it means that $ u_i \in H^j(\Omega)$ $\forall \, j < k-1$. We will show that $ u_i \in H^{j+1}(\Omega)$ $\forall \, j < k-1$. 
 
 Indeed, by dint of the inductive hypothesis, we deduce that
 if $h \in  C^{k-1}(\overline{\Omega})$ then $$ \Delta(h u_i) = \Delta h u_i + u_i \Delta u_i + 2 \nabla h \nabla u_i \in H^{j-1}(\Omega),$$ for every $i = 1, 2, 3$.
 Employing Theorem \ref{D2}, we obtain that
 \begin{equation}\label{system4**} 
\left\{
\begin{aligned}
u_1 - g u_4 - f u_3   && \mbox{in} \quad H^{j+1}(\Omega) \,,\\
u_2 +\overline{g} u_3 + f u_4  && \mbox{in} \quad H^{j+1}(\Omega) \,,\\
u_3 +g u_2 + f u_1  && \mbox{in} \quad H^{j+1}(\Omega) \,,\\
u_4 -\overline{g} u_1 - f  u_2  && \mbox{in} \quad H^{j+1}(\Omega) \,.\\
\end{aligned}  
\right.
\end{equation}
% 
 %Taking $ u_i \in H^1(\Omega)$ for every $i = 1, 2, 3$, we obtain that $ u_i \in H^\frac{1}{2}(\partial\Omega)$ due to the trace theorem.
 We recall the notation $D^k$ as in the proof of Lemma \ref{Regularity hs}. It follows that $$  (\partial_1 + i \partial_2) D^{j-1}u_1 - f (\partial_1 + i \partial_2) D^{j-1}u_3 - g (\partial_1 + i \partial_2) D^{j-1}u_4 \in H^{1}(\Omega) \,.$$
 Differentiating both sides of system \eqref{system3}, one has
\begin{equation}\label{system3**} 
\left\{
\begin{aligned}
(\partial_1-i\partial_2) D^{j-1}u_4 + \partial_3 D^{j-1} u_3 &= i(\lambda -m) D^{j-1}u_1 \quad && \mbox{in} \quad H^1(\Omega) \,,\\
(\partial_1+i\partial_2) D^{j-1}u_3 - \partial_3 D^{j-1} u_4 &= i(\lambda -m) D^{j-1}u_2 \quad && \mbox{in} \quad H^1(\Omega) \,,\\
(\partial_1-i\partial_2) D^{j-1}u_2 + \partial_3 D^{j-1}u_1 &= i(\lambda +m) D^{j-1}u_3 \quad && \mbox{in} \quad H^1(\Omega) \,,\\
(\partial_1+i\partial_2) D^{j-1}u_1 - \partial_3 D^{j-1}u_2 &= i(\lambda +m) D^{j-1}u_4 \quad && \mbox{in} \quad H^1(\Omega) \,,\\
\end{aligned}  
\right.
\end{equation}
Repeat the previous arguments, we obtain
\begin{equation}
( n_1 \partial_1 + n_2 \partial_2 + n_3 \partial_3) D^{j-1}u_4 \in H^{1}(\Omega) .
 \end{equation}
Therefore, $ \frac{\partial D^{j-1}u_4}{\partial n} \in H^{\frac{1}{2}}(\partial\Omega)$ due to the trace theorem. It deduces  that
 \begin{equation} 
\left\{
\begin{aligned}
  -\Delta D^{j-1}u_4 &= ((\lambda^2 -m^2) D^{j-1}u_4  
  && \mbox{in} && H^1(\Omega) 
  \,,
  \\
  \frac{\partial D^{j-1}u_4}{\partial n} & \in H^{\frac{1}{2}}(\partial\Omega) 
  \,.
\end{aligned}
\right.
\end{equation}
Hence, $ D^{j-1} u_4 \in H^2(\Omega)$ due to the Neumann elliptic regularity. As a result, $u_4 \in H^{j+1}(\Omega)$.

Using the analogous arguments for $u_1, u_2, u_3$, we have  $u_i \in H^{k-1}(\Omega)$  for every $i = 1, 2, 3$ by induction on $j$. This completes the proof of the theorem.
\end{proof}
\begin{Corollary}
Let $\Omega$ be a $C^\infty$ smooth bounded domain in $\mathbb{R}^3$ then all eigenfunctions of $(T, \Dom(T))$ are in $C^\infty(\overline{\Omega})$.
\end{Corollary}
\begin{proof}
The result is deduced from Theorem \ref{Main} and the Sobolev embedding $\cap_{k \in \mathbb{N}} H^k(\Omega) = C^\infty(\overline{\Omega})$.
\end{proof}

\begin{Proposition}
Let $\Omega$ be a $C^k$-smooth bounded  domain with $k \geq 3$  then the spectrum of $(D, \Dom(D))$ is purely discrete and all eigenfunctions of $(D, \Dom(D))$ are in $ H^{k-1}(\Omega)$.

\end{Proposition}

\begin{proof}
By dint of Lemma \ref{Lemma-self}, $(D, \Dom(D))$ is self-adjoint and $\sigma_{\textup{ess}}^{\Omega}(D)= \sigma_{\textup{ess}}^{\Omega}(T)$. Moreover, the spectrum of $(T,\Dom(T)$ is purely discrete. Therefore, the spectrum of $(D,\Dom(D))$ is also purely discrete.

The regularity of eigenfunctions of $(D,\Dom(D))$ can be deduced by employing the analogous arguments as in the proof of Theorem \ref{Main}.
\end{proof}

\begin{Remark}
In the case of unbounded smooth domains in $\mathbb{R}^3$ with bounded boundaries, the spectrum of $(T, \Dom(T))$ is purely essential \cite{J.B} and we have not provided any regularity result for the eigenfunctions of the Dirac operator in such domains. It differs from that achieved in two dimensions. The difference arises from the necessity of incorporating the inhomogenous Neumann regularity within the proof of the eigenfunction regularity in three dimensions. 
\end{Remark}
 
Regarding the infinite-mass  Dirac operators employed on three-dimensional polygonal boxes, the self-adjointness is still an open problem even for the case of cubes, despite the validity of the regularity result in $H^\frac{1}{2}$-setting \cite{Behrndt}.

% \subsection*{Data Availability}
 
% \newpage
\subsection*{Declaration of competing interest
}
No  conflict of interest was reported by the author.

\subsection*{Acknowledgment} 
The author was partially supported by the EXPRO grant No.~20-17749X
of the Czech Science Foundation.

%\newpage
%\vfill
%--------------%
% BIBLIOGRAPHY %
%--------------%
%
%\addcontentsline{toc}{section}{References}
%\bibliography{bib}
%\bibliographystyle{amsplain}
%

%\noteD{Some of the references are not cited in the text.}%

%\noteD{References should be alphabetically sorted.
%The style should be unified}%
%\noteD{Use bibTeX.}%
%\noteD{My paper with Briet should be mentioned.}%

\providecommand{\bysame}{\leavevmode\hbox to3em{\hrulefill}\thinspace}
\providecommand{\MR}{\relax\ifhmode\unskip\space\fi MR }
\providecommand{\MRhref}[2]{%
  \href{http://www.ams.org/mathscinet-getitem?mr=#1}{#2}
}
\providecommand{\href}[2]{#2}

\end{document}